\documentclass{amsart} 
\usepackage{amssymb,amsmath,latexsym,times,tikz,hyperref,mathrsfs,enumitem,soul}

\setstcolor{red}
\setul{0pt}{2pt}

\hypersetup{colorlinks=true, linkcolor=blue, citecolor=magenta}

\numberwithin{equation}{section}

\theoremstyle{plain}
\newtheorem{thm}{Theorem}[section]

\newtheorem{lemma}[thm]{Lemma}

\newtheorem{sublemma}{}[thm]

\newtheorem*{lemma*}{Lemma}

\newcommand{\ba}{\backslash}

\theoremstyle{definition}

\newcommand{\del}{\backslash}

\DeclareMathAlphabet{\mathdutchcal}{U}{dutchcal}{m}{n}

\title[The excluded minors for lattice path polymatroids]{The excluded
  minors for lattice path polymatroids} \author[J.~Bonin]{Joseph
  E.~Bonin} \address[J.~Bonin]
{Department of Mathematics\\ The George Washington University\\
  Washington, D.C. 20052, USA} \email{jbonin@gwu.edu}
\author[C.~Chun]{Carolyn Chun}
\address[C.~Chun]{United States Navel Academy\\
  Mathematics Department\\
  Annapolis, MD, 21402, USA} \email{chun@usna.edu}
\author[T.~Fife]{Tara Fife} \address[T.~Fife]{ School of Mathematical
  Sciences, \\ Queen Mary University of London,\\
  Mile End Road, London E1 4NS, United Kingdom}
\email{fi.tara@gmail.com} \date{\today}

\begin{document}

\begin{abstract}
  We find the excluded minors for the minor-closed class of lattice
  path polymatroids as a subclass of the minor-closed class of Boolean
  polymatroids.  Like lattice path matroids and Boolean polymatroids,
  there are infinitely many excluded minors, but they fall into a
  small number of easily-described types.
\end{abstract}

\maketitle

\section{Introduction}

We consider only polymatroids where the rank $\rho(X)$ of
  each set $X$ is a nonnegative integer, so a \emph{polymatroid} on a
finite set $E$ is a function $\rho: 2^E\to \mathbb{Z}$ that is
\begin{enumerate}
\item \emph{normalized}, that is, $\rho(\emptyset) = 0$,
\item \emph{non-decreasing}, that is, if $A\subseteq B\subseteq E$,
  then $\rho(A)\leq \rho(B)$, and
\item \emph{submodular}, that is,
  $\rho(A\cup B)+\rho(A\cap B)\leq \rho(A) + \rho(B)$ for all
  $A, B\subseteq E$.
\end{enumerate}
Herzog and Hibi \cite{HerzogHibi} treat some equivalent formulations
of polymatroids, which are also called \emph{integer polymatroids} or
\emph{discrete polymatroids}.  We often write the ground set $E$ of
$\rho$ as $E(\rho)$.  For a positive integer $t$, a
\emph{$t$-polymatroid} is a polymatroid $\rho$ on $E$ for which
$\rho(e)\leq t$ for all $e\in E$, or, equivalently, $\rho(A)\leq t|A|$
for all $A\subseteq E$.  Thus, matroids are $1$-polymatroids.  The
definitions of deletion and contraction, when cast for matroids using
the rank function, generalize directly to polymatroids (see Section
\ref{sec:background}), and the notion of a minor carries over
directly.

Let $[k]$ denote the set $\{1,2,\ldots,k\}$.  Let $E$ be a finite set
and let $A_i$, for $i\in[k]$, be (not necessarily distinct) subsets of
$E$.  We get a polymatroid $\rho$ on $E$ by, for $X\subseteq E$,
setting
\begin{equation}\label{eq:bpdef}
  \rho(X) = \bigl|\{i\,:\,X\cap A_i\ne\emptyset\}\bigr|.
\end{equation}
Such polymatroids are \emph{Boolean polymatroids} or \emph{transversal
  polymatroids}.  The class of Boolean polymatroids is minor-closed,
that is, every minor of a polymatroid in this class is also in this
class.

Lattice path polymatroids, introduced by Schweig \cite{Jay}, are
constructed as follows. (See Figure \ref{fig:LPM}.)  Take two lattice
paths $P$ and $Q$ from $(1,0)$ to $(n,k)$, where $P$ never rises above
$Q$.  These paths bound a region of the plane.  We label each north
step (a segment from $(i,j)$ to $(i,j+1)$ where $i$ and $j$ are
integers) in this region by its first coordinate.  For
each
$i\in [k]$, let $A_i$ be the set of labels on the north steps in row
$i$ of this diagram, with $i=1$ indexing the lowest row.  The
polymatroid $\rho$ on $E=[n]$ given by equation (\ref{eq:bpdef}) is
the \emph{lattice path polymatroid determined by the paths $P$ and
  $Q$}.  A \emph{lattice path polymatroid} is any polymatroid that is
isomorphic to such a polymatroid from paths.

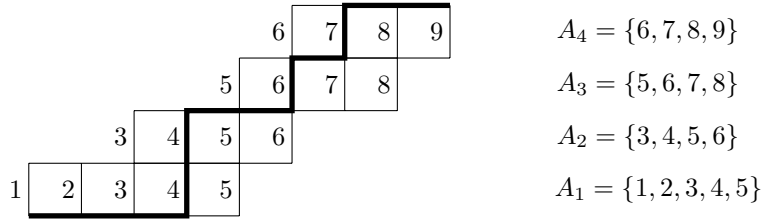
\begin{figure}
  \centering
\begin{tikzpicture}[scale=0.7]
\draw (1,0) grid (5,1); 
\draw (3,1) grid (6,2); 
\draw (5,2) grid (8,3);
\draw (6,3) grid (9,4); 
\draw[line width = 2 pt] (1,0) -- (4,0) -- (4,2) -- (5,2) --
(6,2)--(6,3) -- (7,3)--(7,4)--(9,4); 
\draw (0.75,0.5) node {$1$};
\draw (1.75,0.5) node {$2$};
\draw (2.75,0.5) node {$3$};
\draw (3.75,0.5) node {$4$};
\draw (4.75,0.5) node {$5$};
\draw (2.75,1.5) node {$3$};
\draw (3.75,1.5) node {$4$};
\draw (4.75,1.5) node {$5$};
\draw (5.75,1.5) node {$6$};
\draw (4.75,2.5) node {$5$};
\draw (5.75,2.5) node {$6$};
\draw (6.75,2.5) node {$7$};
\draw (7.75,2.5) node {$8$};
\draw (5.75,3.5) node {$6$};
\draw (6.75,3.5) node {$7$};
\draw (7.75,3.5) node {$8$};
\draw (8.75,3.5) node {$9$};

\draw (12.75,3.5) node {$A_4=\{6,7,8,9\}$};
\draw (12.75,2.5) node {$A_3=\{5,6,7,8\}$};
\draw (12.75,1.5) node {$A_2=\{3,4,5,6\}$};
\draw (12.95,0.5) node {$A_1=\{1,2,3,4,5\}$};
\end{tikzpicture}
\caption{An example of the labeling of north steps and the sets of
  interest in the construction of a lattice path polymatroid.}
  \label{fig:LPM}
\end{figure}

Lattice path polymatroids have many interesting properties
\cite{lppexchange,Jay,Jay2}.  Most lattice path matroids \cite{lpm1}
are not lattice path polymatroids, but the two structures have much in
common.  Lattice path polymatroids form a minor-closed class of
Boolean polymatroids. In this paper, we find the Boolean polymatroids
that are excluded minors for the class of lattice path polymatroids.
That is, we find the set $\mathcal{E}$ of Boolean polymatroids for
which a Boolean polymatroid $\rho$ is a lattice path polymatroid if
and only if no member of $\mathcal{E}$ is a minor of $\rho$.  The set
$\mathcal{E}$, combined with the set of excluded minors for Boolean
polymatroids, which was found by Mat\'u\v{s} \cite{Matus}, gives a
complete characterization of lattice path polymatroids. (Note however
that, as we explain in Section \ref{sec:background}, each member of
$\mathcal{E}$ is a proper minor of infinitely many of the excluded
minors that Mat\'u\v{s} identified.)  Lattice path matroids
\cite{lpmexclmin} and Boolean polymatroids \cite{Matus} have
infinitely many excluded minors, and the same is true of lattice path
polymatroids: the set $\mathcal{E}$ is infinite.  There are nine types
of excluded minors in $\mathcal{E}$, and each type includes infinitely
many polymatroids: four of the types are made up of polymatroids on
three elements, four other types are made up of polymatroids on four
elements, and one type is made up of polymatroids on any number of
elements greater than two.

In Section \ref{sec:background}, we review the relevant background on
polymatroids, minors, Boolean polymatroids, and lattice path
polymatroids; in particular, we state the excluded-minor
characterization of Boolean polymatroids by Mat\'u\v{s} \cite{Matus}
and explain how it relates to our results.  In Section
\ref{sec:cycles}, we identify the type of excluded minors that can
have any number of elements three or greater.  In Section
\ref{sec:34}, we identify the types of excluded minors on three or
four elements.  In Section \ref{sec:main}, specifically Theorem
\ref{excludedminors}, we prove that the collection of excluded minors
identified in Sections \ref{sec:cycles} and \ref{sec:34}
is complete.

\section{Background}\label{sec:background}

Let $\rho$ be a polymatroid on $E$.  For $A\subseteq E$, the
\emph{deletion $\rho_{\del A}$} and \emph{contraction $\rho_{/A}$},
which are polymatroids on $E-A$, are given by
$\rho_{\del A}(X) = \rho(X)$ and
$\rho_{/A}(X) = \rho(X\cup A)-\rho(A)$ for all $X\subseteq E-A$. The
\emph{minors} of $\rho$ are the polymatroids of the form
$(\rho_{\del A})_{/B}$ (equivalently, $(\rho_{/B})_{\del A}$) for
disjoint subsets $A$ and $B$ of $E$.

If $M_1,M_2,\ldots,M_k$ are matroids on $E$, then the function
$\rho:2^E\to\mathbb{Z}$ given by
$$\rho(X) = r_{M_1}(X)+ r_{M_2}(X)+\cdots+ r_{M_k}(X),$$ for
$X\subseteq E$, is a polymatroid on $E$.  We write
$\rho = r_{M_1}+ r_{M_2}+\cdots+ r_{M_k}$ to express this more
briefly.  (Not all polymatroids have such decompositions into matroid
rank functions.)  It is easy to check that if
$\rho = r_{M_1}+ r_{M_2}+\cdots+ r_{M_k}$, then
$$\rho_{\del A} = r_{M_1\del A}+ r_{M_2\del A}+\cdots+ r_{M_k\del A}
\quad\text{and}\quad \rho_{/A} = r_{M_1/A}+ r_{M_2/A}+\cdots+
r_{M_k/A}.$$

It is easy to see that the Boolean polymatroid $\rho$ on $E$ given by
equation (\ref{eq:bpdef}), using the nonempty sets $A_i$
for $i\in[k]$, can be written as the sum of $k$ rank-$1$ matroids on
$E$, namely, 
\begin{equation}\label{eq:decomp}
  \rho=r_{U_{1,A_1}\oplus U_{0,E-A_1}}+ r_{U_{1,A_2}\oplus
    U_{0,E-A_2}}+\cdots+r_{U_{1,A_k}\oplus U_{0,E-A_k}},
\end{equation}
where, adapting the usual notation $U_{r,n}$ for uniform matroids,
$U_{r,A}$ is the rank-$r$ uniform matroid on the set $A$.

For completeness, we next prove a
result that, while not new, is crucial to our work: the Boolean
polymatroid $\rho$ determines the nonempty sets $A_i$ (but, of course,
not how these sets are indexed).

\begin{lemma}\label{lem:AisUnique}
  Let $A_1,A_2,\ldots,A_k$ be (not necessarily
    distinct) nonempty subsets of a set $E$ and let $\rho$ be the
  Boolean polymatroid on $E$ given by equation
  \emph{(\ref{eq:decomp})}.  For each nonempty subset
  $Z$ of $E$, the multiplicity of $Z$ in the list $A_1,A_2,\ldots,A_k$
  can be computed from $\rho$.
\end{lemma}

\begin{proof}
  For each $e\in E$, we define a property $p_e$ that each integer
  $i\in[k]$ may or may not have: $i$ has property $p_e$ if
  $e\not\in A_i$.  Fix $X\subseteq E$.  Now $i\in[k]$ has the
  properties $p_e$ with $e\in X$ and no other properties if and only
  if $A_i=E-X$.  The number of $i\in[k]$ that have all the properties
  $p_e$ with $e\in X$, and maybe more, is $k-\rho(X)$.  By
  inclusion-exclusion, it follows that the number of $i\in[k]$ that
  have the properties $p_e$ with $e\in X$ and no others is
  $$\sum_{Y\,:\,X\subseteq Y\subseteq E}(-1)^{|Y-X|}\bigl(k-\rho(Y)\bigr).$$  
  Since this is the number of $i$ for which $A_i=E-X$, the lemma
  follows.
\end{proof}

This argument is closely related to the characterization of Boolean
polymatroids by inequalities involving the rank function in
Mat\'u\v{s} \cite[Lemma 2]{Matus} (see Theorem
\ref{thm:booleaninequalties} below).

Lemma \ref{lem:AisUnique} leads to a well-defined (up to relabeling
the indices) notion of support: given $E$ and $A_i$, for $i\in [k]$,
and $\rho$ given by equation (\ref{eq:decomp}), the \emph{support}
$s(e)$ of $e\in E$ is 
$$s(e) =\{i\,:\,i\in [k] \text{ and } e\in A_i\}.$$ 
The \emph{support of $\rho$} is $\cup _{e\in E} s(e)$, which is $[k]$
if no set $A_i$ is empty.  For example, for the lattice path
polymatroid in Figure \ref{fig:LPM}, we have $s(1)=s(2)=\{1\}$,
$s(3)=s(4)=\{1,2\}$, $s(5)=\{1,2,3\}$, $s(6)=\{2,3,4\}$,
$s(7)=s(8)=\{3,4\}$, and $s(9)=\{4\}$.  Note that for any Boolean
polymatroid $\rho$ on $E$ and all $X\subseteq E$,
\begin{equation}\label{eq:rankbysupport}
\rho(X) = \Bigl|\bigcup_{e\in X}s(e)\Bigr|.
\end{equation}

There is a natural linear order on $[k]$, namely, $1<2<\cdots< k$, as
well as on the ground set $E$ of a lattice path polymatroid determined
by two paths, but this is not so for the excluded minors.  All orders
(or orderings) that we consider are linear orders (also known as total
orders).  They can be denoted by listing the elements from least to
greatest, as in $x_1,x_2,\ldots,x_t$, or via the conventional symbol
for order, as in $x_1<x_2<\cdots<x_t$.

Since lattice path polymatroids in general are just isomorphic to
those determined by two paths, it follows that $\rho$ is a lattice
path polymatroid if and only if there is some ordering of $[k]$ and
some ordering $e_1,e_2,\ldots,e_n$ of $E$ so that,
\begin{itemize}
\item[(S1)] for each $i\in [n]$, the support $s(e_i)$ is an interval
  $[a_i,b_i]$ in the order on $[k]$, and,
\item[(S2)] in that order on $[k]$, we have
  $a_1\leq a_2\leq\cdots\leq a_n$ and $b_1\leq b_2\leq\cdots\leq b_n$.
\end{itemize} 
Such an ordering of $E$ is a \emph{lattice path ordering of $\rho$},
and such an ordering of $\left[ k\right]$ is a \emph{lattice path
  ordering of the support of $\rho$}.
Given sets $Z_1,Z_2,\dots ,Z_{\ell}$ that partition $\left[ k\right]$
such that a lattice path ordering of the support of $\rho$ can be
obtained by taking any ordering of the elements of $Z_1$ followed by
any ordering of $Z_2$ and so on, finally ending with any ordering of
the elements of $Z_{\ell}$, we also refer to $Z_1,Z_2,\dots ,Z_{\ell}$
as a \emph{lattice path ordering of the support of $\rho$}.  For ease
of notation, we also consider any sequence obtained from
$Z_1,Z_2,\dots ,Z_{\ell}$ by adding copies of the empty set to be a
\emph{lattice path ordering of the support of $\rho$}.
  
By equation (\ref{eq:rankbysupport}), knowing the support of each
element $e_i$ in a Boolean polymatroid $\rho$ determines $\rho$.
Different elements in a Boolean polymatroid may have the same support,
so the map taking each element to its support need not be injective.
Also, the support $s(e_i)$ depends on the order in which we write the
sets $A_1,A_2,\ldots,A_k$, and changing the order by a permutation
$\pi$ just replaces each support by its image under $\pi$.  To
determine whether $\rho$ is a lattice path polymatroid, we need to
determine whether a lattice path ordering of $\rho$ exists, and its
existence goes hand in hand with the existence of a compatible lattice
path ordering of the support.

From equation (\ref{eq:decomp}) and the discussion of minors before
it, it follows that, in terms of supports, deleting $e$ from a Boolean
polymatroid corresponds to deleting its support set $s(e)$, while
contracting $e$ from a Boolean polymatroid corresponds to deleting
$s(e)$ and, for each element $f\in E-e$, replacing $s(f)$ by
$s(f)-s(e)$ since contracting a nonloop in a rank-$1$ matroid yields a
rank-$0$ matroid.  It now easily follows that minors of lattice path
polymatroids are lattice path polymatroids, but we get much more: if
$\rho$ is a Boolean polymatroid and $\rho_{\del e}$ is a lattice path
polymatroid, then so is $\rho_{/e}$ since, given a linear order on
$[k]$ in which the support sets for $\rho_{\del e}$ are intervals that
satisfy (S2), for the induced linear order on $[k]-s(e)$, the support
sets for $\rho_{/e}$ are intervals that satisfy (S2).  Thus, we have
the following lemma, which is highly atypical for minor-closed classes
of polymatroids.

\begin{lemma}\label{lem:deletionsonly}
  If $\rho$ is a Boolean polymatroid that is a not a lattice path
  polymatroid, then $\rho$ is an excluded minor for the class of
  lattice path polymatroids if and only if each single-element
  deletion is a lattice path polymatroid.
\end{lemma}

As noted above, the class of lattice path polymatroids is a
minor-closed subclass of the class of Boolean polymatroids.  The
excluded-minors for the class of Boolean polymatroids were found by
Mat\'u\v{s}.  The next theorem recasts  \cite[Theorem 3]{Matus}.

\begin{thm}\label{matusex}
  The excluded minors for Boolean polymatroids are the polymatroids of
  the form
  \begin{equation}\label{eq:booleanexminform}
    -c_E\,r_{U_{1,E}} +\sum_{X\,:\,\emptyset\neq X\subsetneq E}c_X\,
    r_{U_{1,X}\oplus U_{0,E-X}}
  \end{equation}
  where $|E|\geq 3$, $c_E$ is a positive integer, all coefficients
  $c_X$ are nonnegative integers, and $c_X\geq c_E$ if $|X|=|E|-1$.
\end{thm}

Thus, there are infinitely many excluded minors for Boolean
polymatroids.  To translate this result into an infinite list of the
excluded minors, one must take isomorphism into account since acting
on $2^E$ with a permutation of $E$ yields isomorphic excluded minors
that all have the form in (\ref{eq:booleanexminform}).  It follows
from Theorem \ref{matusex} that each Boolean polymatroid $\rho'$ on
$E'$ is a minor of infinitely many of these excluded minors: take any
superset $E$ of $E'$ with $|E-E'|\geq 2$ and let $\rho$ be as in
(\ref{eq:booleanexminform}) where, for $X\subseteq E'$, the
coefficient $c_X$ is the number of copies of
$r_{U_{1,X}\oplus U_{0,E'-X} }$ when $\rho'$ is written as in equation
(\ref{eq:decomp}); then $\rho_{/E-E'}=\rho'$.  Similarly, one can
construct excluded minors for Boolean polymatroids that have $\rho'$
as a deletion.  Given a characterization of a proper minor-closed
class of Boolean polymatroids by its excluded minors relative to the
class of Boolean polymatroids, there is an excluded-minor
characterization of the class relative to the class of all
polymatroids; however, the excluded minors relative to the class of
Boolean polymatroids may, in some cases, be more illuminating.  Of
course, each Boolean polymatroid that is an excluded minor relative to
the class of Boolean polymatroids is also an excluded minor relative
to the class of all polymatroids.

Besides using the excluded minors, there is another way
to determine whether a polymatroid is Boolean.  The next result is
part of \cite[Lemma 1]{Matus}.

\begin{thm}\label{thm:booleaninequalties}
  A polymatroid $\rho$ on $E$ is Boolean if and only if, for all
  $X\subseteq E$, 
  $$\sum_{Y\,:\,X\subseteq Y\subseteq
    E}(-1)^{|Y-X|}\bigl(\rho(E)-\rho(Y)\bigr)\geq 0.$$
\end{thm}

Given these factors, it is reasonable to seek characterizations of
proper minor-closed classes of Boolean polymatroids by their excluded
minors relative to the class of Boolean polymatroids.  This is exactly
what we do for lattice path polymatroids.

\section{Boolean cycles and a property of lattice path
  orderings}\label{sec:cycles}

We start with a necessary condition for an ordering on the ground set
$E$ of a Boolean polymatroid $\rho$ to be a lattice path ordering of
$E$, and so for $\rho$ to be lattice path.

\begin{lemma}\label{lem:S3}
  Let $\rho$ be a Boolean polymatroid on $E$.  If $\rho$ is a lattice
  path polymatroid on $E$ and $e_1,e_2,\dots ,e_n$ is a lattice path
  ordering of $\rho$, then the following property holds:
  \begin{itemize}
  \item[\emph{(S3)}] if 
     $S\ne\emptyset$ and $S\subseteq s(e_h)$ for
    at least one $h\in [n]$, then the elements whose supports contain 
    $S$ are $e_i,e_{i+1},\dots ,e_{j-1}$, and $e_j$ for some $i$ and
    $j$ with $1\leq i\leq j\leq n$.
  \end{itemize}
\end{lemma}

\begin{proof}
  By relabeling if needed, we can take the usual order $1,2,\ldots,k$
  on the support of $\rho$ and let $s(e_t)$ be the interval
  $[a_t,b_t]$ in $[k]$.  Let $i\in[k]$ be least with
  $S\subseteq s(e_i)$, and let $j\in[k]$ be greatest with
  $S\subseteq s(e_j)$.  The result holds if $i=j$, so assume that
  $i<j$.  Now $S\subseteq s(e_i)\cap s(e_j)=[a_j,b_i]$.  Since
  $a_i\leq a_{i+1}\leq \cdots \leq a_j$ and
  $b_i\leq b_{i+1}\leq \cdots \leq b_j$, we have
  $[a_j,b_i]\subseteq s(e_h)$ for all $h$ with $i\leq h\leq j$, as
  needed.
\end{proof}

Recall that, by equation (\ref{eq:rankbysupport}), a Boolean
polymatroid can be given by its supports.  That is how we will define
each of the excluded minors for lattice path polymatroids.

We first treat a family of excluded minors that, for each integer
$n\geq 3$, contains infinitely many polymatroids with $n$ elements.
We call a polymatroid in this family a \emph{Boolean cycle}, or a
\emph{Boolean $n$-cycle}, where $n$ is the number of elements.  Let
$n\geq 3$.  The supports of the elements $e_1,e_2,\ldots,e_n$ of a
\emph{Boolean $n$-cycle} have the form
$s(e_i)=Z_{2i-1}\cup Z_{2i}\cup Z_{2i+1}$, where the indices are taken
modulo $2n$ (thus, $s(e_n)=Z_{2n-1}\cup Z_{2n}\cup Z_1$), and
\begin{itemize}[leftmargin=*]
\item $Z_1,Z_2,\ldots,Z_{2n}$ are pairwise disjoint, and
\item for all $i\in [n]$, the set $Z_{2i-1}$ is nonempty.
\end{itemize}
The sets $Z_{2i}$ may be empty.  The example of smallest rank and size
is three coplanar lines, that is, $e_1$, $e_2$, $e_3$ each have rank
$2$ and any set of two or three of them has rank $3$.  As with the
description of the excluded minors for Boolean polymatroids in Theorem
\ref{matusex}, the definition of Boolean $n$-circuits on $E$ yields
many isomorphic polymatroids; we are, of course, interested in these
polymatroids up to isomorphism.

We now prove that each Boolean $n$-cycle is an excluded minor.

\begin{lemma}\label{cycles}
  Let $\rho$ be a Boolean $n$-cycle for some integer $n\geq 3$.  Then
  $\rho$ is an excluded minor for the class of lattice path
  polymatroids.
\end{lemma}

\begin{proof}
  Suppose $\rho$ is lattice path.  Since the supports of only two
  elements contain $Z_3$, namely $e_1$ and $e_2$, these elements are
  consecutive in the lattice path ordering of $\rho$.  Likewise, since
  $Z_{2i-1}$ is only contained in $s(e_{i-1})$ and $s(e_i)$ for all
  $i\in\{3,4,\dots ,n\}$, the ordering on the elements of $\rho$ must
  be $e_1,e_2,\dots ,e_n$.  However $s(e_1)\cap s(e_n)= Z_1$ and
  $Z_1\nsubseteq s(e_2)$, so (S3) fails, and so $\rho$ is not lattice
  path.

  Take $i\in [n]$.  Then
  $e_{i+1},e_{i+2},\dots ,e_n,e_1,e_2,\dots ,e_{i-1}$ is a lattice
  path ordering of $\rho_{\ba e_i}$.  Thus, $\rho$ is an excluded
  minor for the class of lattice path polymatroids by
  Lemma~\ref{lem:deletionsonly}.
\end{proof}

\section{Excluded minors with three or four elements}\label{sec:34}

In this section, we define the types of excluded minors that are not
Boolean cycles.  Each of these types has three or four elements.  We
also prove that our collection of three-element excluded minors is
complete.

We first note that each polymatroid on a set of at most two elements
is a lattice path polymatroid; in particular, it is a Boolean
polymatroid.  This is because, for a polymatroid $\rho$ on $E=\{e,f\}$
with $k=\rho(E)$, we can let $s(e)$ be the first $\rho(e)$ elements,
and $s(f)$ the last $\rho(f)$ elements, of $[k]$.

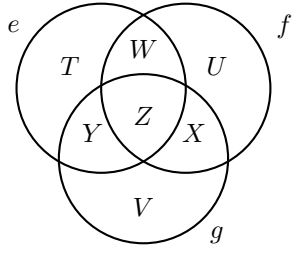
\begin{figure}
\begin{tikzpicture}[scale=0.75]
  \draw[thick] (2.25,3.5) circle (1.5);
  \draw[thick] (3.75,3.5) circle (1.5);
  \draw[thick] (3,2.25) circle (1.5);
  \node () at (1.7,3.9) {$T$};
\node () at (4.3,3.9) {$U$};
\node () at (3,1.4) {$V$};
\node () at (5.5,4.6) {$f$};
\node () at (0.7,4.6) {$e$};
\node () at (4.3,0.9) {$g$};
\node () at (3,3) {$Z$};
\node () at (3,4.2) {$W$};
\node () at (2.1,2.7) {$Y$};
\node () at (3.9,2.7) {$X$};
\end{tikzpicture}

\caption{Venn diagram showing the supports of elements $e$, $f$, and
  $g$ in a Boolean polymatroid, $\rho$.}
\label{3venn}
\end{figure}

We turn to the Boolean excluded minors that have three elements and
are not Boolean $3$-cycles.  They come in types $T_1$, $T_2$, $T_3$,
and $T_4$, and there are infinitely many of each type.  In order to
characterize these four types of Boolean polymatroids, say with
elements $e$, $f$, and $g$, it suffices to identify which of the areas
in the Venn diagram of their supports in Figure~\ref{3venn} are empty,
and which are nonempty.  Note that if $Z=\emptyset$ and each of $W$,
$X$, and $Y$ is nonempty, then $\rho$ is a Boolean $3$-cycle,
independent of whether or not $T$, $U$, or $V$ are empty.

In the following list, we assume that $\rho$ is a Boolean polymatroid
and $E(\rho)=\{e,f,g\}$, where the supports of $e$, $f$, and $g$ are
as shown in Figure~\ref{3venn}.  We view elements of a rank-$k$
Boolean polymatroid in terms of their supports, which are subsets of
$[k]$ and so can be identified with faces of a simplex with $k$
vertices, so the examples that we provide are given as sets of faces
of a simplex.
\begin{itemize}
\item[($T_1$)] The polymatroid $\rho$ is type $T_1$ if
  \begin{itemize}[leftmargin=*]
  \item[$\bullet$] $T$, $U$, $V$, and $Z$ are all nonempty.
  \end{itemize}
\item[$\circ$] Example: three lines in rank four that, on the simplex,
  share a vertex (the vertex does not correspond to an element of the
  polymatroid).
\item[($T_2$)] The polymatroid $\rho$ is type $T_2$ if
  \begin{itemize}[leftmargin=*]
  \item[$\bullet$] $W$, $X$, $Y$, and $Z$ are all nonempty and
  \item[$\bullet$] at least one of $T$, $U$, or $V$ is empty.
  \end{itemize}
\item[$\circ$] An example:  three planes in the four-vertex
  simplex.
\item[($T_3$)] The polymatroid $\rho$ is type
  $T_3$ if
  \begin{itemize}[leftmargin=*]
  \item[$\bullet$] $T$, $U$, $W$, and $Z$ are all nonempty and
  \item[$\bullet$] $V=X=Y=\emptyset$.
  \end{itemize}
\item[$\circ$] Example: two planes of the four-vertex simplex, along
  with one of the two points on the line that the planes share on the
  simplex.
\item[($T_4$)] The polymatroid $\rho$ is type $T_4$ if
  \begin{itemize}[leftmargin=*]
  \item[$\bullet$] $T$, $W$, $Y$, and $Z$ are all nonempty and
  \item[$\bullet$] $X=V=\emptyset$.
  \end{itemize}
\item[$\circ$] The set $U$ may or may not be empty. Example with
  $U=\emptyset$: an element of rank four and two lines on the
  corresponding simplex that share one vertex (the vertex is not in
  the polymatroid).  Example with $U\ne\emptyset$: start with a
  similar configuration in a four-vertex simplex, but embedded in a
  five-vertex simplex, with one of the lines extended to a plane that
  includes the new vertex.
\end{itemize}

\begin{figure}
\begin{tikzpicture}[scale=0.75]
  \draw[thick] (2.25,3.5) circle (1.5);
  \draw[thick] (3.75,3.5) circle (1.5);
  \draw[thick] (3,2.25) circle (1.5);
  \node () at (1.7,3.9) {$T$};
  \node () at (1.2,4.1) {$\bullet$};
  \node () at (4.3,3.9) {$U$};
  \node () at (4.7,4.1) {$\bullet$};
  \node () at (3,1.5) {$V$};
  \node () at (3,1.1) {$\bullet$};
  \node () at (5.5,4.6) {$f$};
  \node () at (0.7,4.6) {$e$};
  \node () at (4.3,0.9) {$g$};
  \node () at (3,3.2) {$Z$};
  \node () at (3,2.6) {$\bullet$};
  \node () at (3,4.2) {$W$};
  \node () at (2.1,2.7) {$Y$};
  \node () at (3.9,2.7) {$X$};
  \node () at (2.2,0.3) {$T_1$};

\begin{scope}
\clip (10.5,2.25) circle (1.5);
\fill[color= black!30] (8.2,0.75) rectangle (12.75,5); 
\end{scope}
 \begin{scope}
 \clip (9.75,3.5) circle (1.5);
 \fill[color= white] (8.2,0.75) rectangle (12.75,5); 
\end{scope}
 \begin{scope}
 \clip (11.25,3.5) circle (1.5);
 \fill[color= white] (8.2,0.75) rectangle (12.75,5); 
 \end{scope}
  \draw[thick] (9.75,3.5) circle (1.5);
  \draw[thick] (11.25,3.5) circle (1.5);
  \draw[thick] (10.5,2.25) circle (1.5);
\node () at (9.2,3.9) {$T$};
\node () at (11.8,3.9) {$U$};
\node () at (10.5,1.4) {$V$};
\node () at (13,4.6) {$f$};
\node () at (8.2,4.6) {$e$};
\node () at (11.8,0.9) {$g$};
\node () at (10.5,3.2) {$Z$};
\node () at (10.5,2.6) {$\bullet$};
\node () at (10.5,4.1) {$W$};
\node () at (10.5,4.5) {$\bullet$};
\node () at (9.6,2.7) {$Y$};
\node () at (9.3,2.3) {$\bullet$};
\node () at (11.4,2.7) {$X$};
\node () at (11.7,2.3) {$\bullet$};
 \node () at (9.7,0.3) {$T_2$};
\end{tikzpicture}

\vspace{20pt}

\begin{tikzpicture}[scale=0.75]
\begin{scope}
\clip (3,2.25) circle (1.5);
\fill[color= black!30] (0.75,0.75) rectangle (5.25,5); 
\end{scope}
 \begin{scope}
 \clip (3.75, 3.5) circle (1.5);
 \clip (2.25, 3.5) circle (1.5);
 \fill[color= white]  (0.75,0.75) rectangle (5.25,5); 
 \end{scope}
  \draw[thick] (2.25,3.5) circle (1.5);
  \draw[thick] (3.75,3.5) circle (1.5);
  \draw[thick] (3,2.25) circle (1.5);
  \node () at (1.7,3.9) {$T$};
  \node () at (1.2,4.1) {$\bullet$};
\node () at (4.3,3.9) {$U$};
  \node () at (4.7,4.1) {$\bullet$};
\node () at (3,1.4) {$V$};
\node () at (5.5,4.6) {$f$};
\node () at (0.7,4.6) {$e$};
\node () at (4.3,0.9) {$g$};
\node () at (3,3.2) {$Z$};
\node () at (3,2.6) {$\bullet$};
\node () at (3,4.2) {$W$};
\node () at (3,4.5) {$\bullet$};
\node () at (2.1,2.7) {$Y$};
\node () at (3.9,2.7) {$X$};
 \node () at (2.2,0.3) {$T_3$};

\begin{scope}
\clip (10.5,2.25) circle (1.5);
\fill[color= black!30] (8.2,0.75) rectangle (12.75,5); 
\end{scope}
 \begin{scope}
 \clip (9.75,3.5) circle (1.5);
 \clip (10.5,2.25) circle (1.5);
 \fill[color= white] (8.2,0.75) rectangle (12.75,5); 
\end{scope}
  \draw[thick] (9.75,3.5) circle (1.5);
  \draw[thick] (11.25,3.5) circle (1.5);
  \draw[thick] (10.5,2.25) circle (1.5);
\node () at (9.2,3.9) {$T$};
  \node () at (8.7,4.1) {$\bullet$};
\node () at (11.8,3.9) {$U$};
\node () at (10.5,1.4) {$V$};
\node () at (13,4.6) {$f$};
\node () at (8.2,4.6) {$e$};
\node () at (11.8,0.9) {$g$};
\node () at (10.5,3.2) {$Z$};
\node () at (10.5,2.6) {$\bullet$};
\node () at (10.5,4.1) {$W$};
\node () at (10.5,4.5) {$\bullet$};
\node () at (9.6,2.7) {$Y$};
\node () at (9.3,2.3) {$\bullet$};
\node () at (11.4,2.7) {$X$};
 \node () at (9.7,0.3) {$T_4$};
\end{tikzpicture}
  
\caption{Venn diagrams showing the supports of elements $e$, $f$, and $g$
  in Boolean polymatroids of types $T_1$, $T_2$, $T_3$, and $T_4$.  Areas
  shaded gray indicate that those sets are empty.  A point in an area
  indicates that that set is nonempty.
  The other sets may or may not be empty.  }
\label{3venn_ex}
\end{figure}
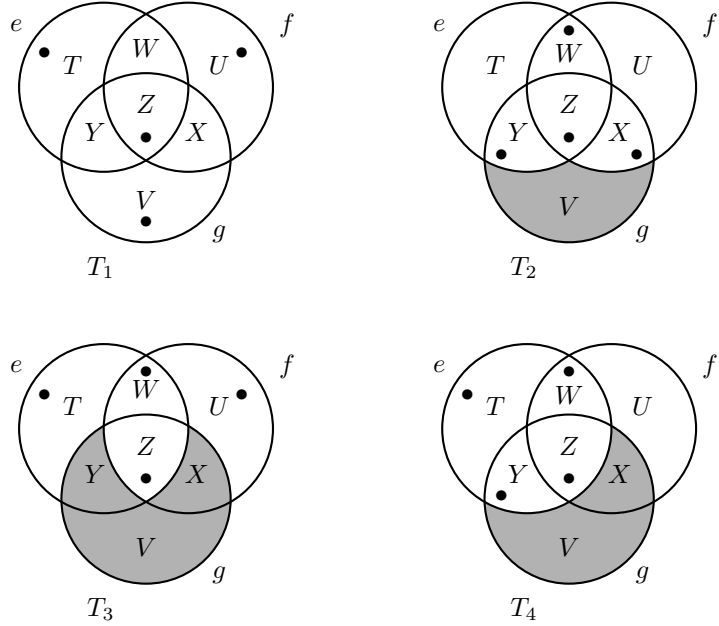

  Using equation (\ref{eq:rankbysupport}), it is easy to
  translate these descriptions into the values of $\rho$.  For
  instance, a polymatroid $\rho$ on $E=\{e,f,g\}$ is type $T_1$ if and
  only if there are positive integers $z$, $t$, $u$, and $v$, and
  nonnegative integers $w$, $x$, and $y$, for which
  $$\rho(e) = z+t+w+y, \qquad\rho(f) = z+u+w+x, \qquad \rho(g) =
  z+v+x+y,$$
  $$\rho(e,f) = z+t+u+w+x+y, \quad\qquad\rho(e,g) = z+t+v+w+x+y,$$
  $$
  \rho(f,g) = z+u+v+w+x+y, \qquad \rho(E)=z+t+u+v+w+x+y.$$ 

We next show that the polymatroids of these four types, along with
Boolean $3$-cycles, are the $3$-element Boolean excluded minors for
lattice path polymatroids.

\begin{lemma}\label{3element}
  A Boolean polymatroid $\rho$ is an excluded minor for the class of
  lattice path polymatroids, where $|E(\rho)|=3$, if and only if
  $\rho$ is type $T_1$, $T_2$, $T_3$, or $T_4$ or $\rho$ is a Boolean
  $3$-cycle.
\end{lemma}

\begin{proof}
  It is straightforward to check that if $\rho$ is type $T_1$, $T_2$,
  $T_3$, or $T_4$ or a Boolean $3$-cycle, then $\rho$ is not lattice
  path.  Since every two-element polymatroid is lattice path, $\rho$
  is therefore an excluded minor for the class of lattice path
  polymatroids.

  Assume then that $\rho$ is a three-element excluded minor for the
  class of lattice path polymatroids.  Let $E=\{e,f,g\}$ where the
  supports of these elements are shown in Figure~\ref{3venn}.  Suppose
  first that $Z=\emptyset$.  If $W=\emptyset$, then $T,Y,V,X,U$ is a
  lattice path ordering of the support of $\rho$, which is a
  contradiction.  By symmetry, none of $W$, $X$, $Y$ is empty, and
  $\rho$ is a Boolean $3$-cycle.

  Now assume that $Z\neq\emptyset$.  If none of $T$, $U$, $V$ is
  empty, then $\rho$ is type $T_1$, so we assume not; say
  $V=\emptyset$.  If each of $W$, $X$, $Y$ is nonempty, then $\rho$ is
  type $T_2$, so we assume not.  If $W=\emptyset$, then $T,Y,Z,X,U$ is
  a lattice path ordering of the support of $\rho$, which is a
  contradiction.  Therefore $W\neq \emptyset$ and, without loss of
  generality, $X=\emptyset$.  If $T=\emptyset$, then $Y,Z,W,U$ is a
  lattice path ordering of the support of $\rho$, which is a
  contradiction.  Hence $T\neq\emptyset$.  If $Y$ is nonempty, then
  $\rho$ is type $T_4$, so we assume not.  If $U=\emptyset$, then
  $T,W,Z$ is a lattice path ordering of the support of $\rho$, which
  is a contradiction.  Therefore $U\neq\emptyset$ and $\rho$ is type
  $T_3$.
\end{proof}

We now present all four-element Boolean polymatroids, other than
Boolean $4$-cycles, that are excluded minors for the class of lattice
path polymatroids.  There are four types, and infinitely many of each
type.  Let $\rho$ be a Boolean polymatroid on four elements.

\begin{itemize}
\item[($F_1$)] The polymatroid $\rho$ is type $F_1$ if the supports
  have the form $A\cup B$, $A\cup B\cup C$, $B\cup C\cup D$, and
  $B\cup D$ where
  \begin{itemize}[leftmargin=*]
  \item[$\bullet$] $A$, $B$, $C$, and $D$ are pairwise disjoint and
  \item[$\bullet$] $\emptyset\notin\{A,B,C\}$.   
  \end{itemize}
\item[$\circ$] Example with $D\ne \emptyset$: two planes in the
  four-vertex simplex along with a line in exactly one of those planes
  and the line in just the other plane that is coplanar with the first
  line.  Example with $D=\emptyset$: a vertex in a triangle, the two
  lines containing that vertex, and the triangle.
\item[($F_2$)] The polymatroid $\rho$ is type $F_2$ if the supports
  have the form $A$, $B$, $C$, and $D$ where
  \begin{itemize}[leftmargin=*]
  \item[$\bullet$] $A$, $B$, and $C$ are pairwise disjoint, and
  \item[$\bullet$] $A\cap D$, $B\cap D$, and $C\cap D$ are all nonempty.  
  \end{itemize}
\item[$\circ$] Example: three noncolinear points
    along with the plane that contains them.
  \item[($F_3$)] The polymatroid $\rho$ is type $F_3$ if the supports
    have the form $A$, $A'$, $B$, and $C$ where
  \begin{itemize}[leftmargin=*]
  \item[$\bullet$] $A'\subseteq A$ and $A'\neq A$, and 
  \item[$\bullet$] $B\cap C=\emptyset$, and 
  \item[$\bullet$] both $B$ and $C$ are disjoint from $A-A'$, and
    neither is disjoint from $A'$.
  \end{itemize}
\item[$\circ$] Example: two points, the line they span, and a plane
  that contains that line.
\item[($F_4$)] The polymatroid $\rho$ is type $F_4$ if the supports
  have the form $A\cup B\cup C$, $B\cup C\cup D$, $C\cup D\cup E$, and
  $A\cup C\cup E$ where
  \begin{itemize}[leftmargin=*]
  \item[$\bullet$] $A$, $B$, $C$, $D$, and $E$ are all pairwise
    disjoint and nonempty.
  \end{itemize}
\item[$\circ$] Example: take four of the six lines in a four-vertex
  simplex, with the pair omitted being skew; embed this simplex in a
  five-vertex simplex and replace each line by the plane that is
  spanned by it and the new vertex; these four planes are the elements
  of the polymatroid.
\end{itemize}

\begin{figure}
\begin{tikzpicture}[scale=0.75]
  \draw[thick] (1.2,2.25) circle (0.75);
  \draw[thick] (4.4,3.2) circle (0.75);
  \draw[thick] (4.4,1.3) circle (0.75);
  \draw[thick] (3,2.25) circle (1.5);
  \node () at (2,3.8) {$D$};
  \node () at (5.5,3.2) {$B$};
  \node () at (5.5,1.3) {$C$};
  \node () at (0.7,3.2) {$A$};
  \node () at (3,0.2) {$F_2$};
  \node () at (1.7,2.25) {$\bullet$};
  \node () at (4,3) {$\bullet$};
  \node () at (4,1.5) {$\bullet$};
  
 \begin{scope}
   \clip (8.7,2.25) circle (1.5);
   \clip (10.5,2.25) circle (1.5);
   \fill[color= black!30] (10.5,2.25) circle (1.5);
 \end{scope}
 \begin{scope}
   \clip (12.3,2.25) circle (1.5);
   \clip (10.5,2.25) circle (1.5);
  \fill[color= black!30] (10.5,2.25) circle (1.5);
 \end{scope}
 \fill[color= white] (10.5,2.25) circle (0.75);
 \draw[thick] (8.7,2.25) circle (1.5);
  \draw[thick] (12.3,2.25) circle (1.5);
  \draw[thick] (10.5,2.25) circle (0.75);
  \draw[thick] (10.5,2.25) circle (1.5);
\node () at (7.6,3.8) {$B$};
\node () at (13.4,3.8) {$C$};
\node () at (10.5,2.25) {$A'$};
\node () at (10.5,4.1) {$A$};
\node () at (10.5,3.3) {$\bullet$};
\node () at (10,2.25) {$\bullet$};
\node () at (11,2.25) {$\bullet$};
 \node () at (10.5,0.2) {$F_3$};
\end{tikzpicture}

\caption{Venn diagrams showing the supports of elements in a Boolean
  polymatroid.  Areas shaded gray indicate that those sets are empty.
  A point in an area indicates that that set is nonempty.  Areas that
  are not gray and contain no point may be empty or nonempty.}
\label{4venn_ex}
\end{figure}
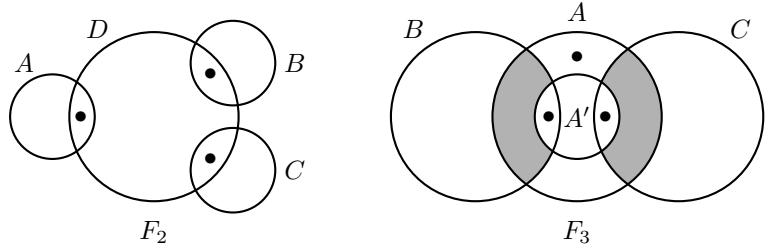

The Venn diagrams of the supports of the element for types $F_2$ and
$F_3$ are shown in Figure~\ref{4venn_ex}, where a point appears in an
area for which the set is known to be nonempty, and a darkened area
indicates that the set is empty.  It is straightforward to
  check that if $\rho$ is type $F_1$, $F_2$, $F_3$, or $F_4$, then
  $\rho$ is not lattice path, but all of its proper minors are.

\section{Proof of the main result}\label{sec:main}

In this section we prove our main result: when considered as a
subclass of the class of Boolean polymatroids, the set of excluded
minors given in the previous two sections for the class of lattice
path polymatroids is complete.

We start with an easy result that is worth noting.

\begin{lemma}\label{distinctsupport}
  If $\rho$ is a Boolean polymatroid that is an excluded minor for the
  class of lattice path polymatroids, and $e$ and $f$ are distinct
  elements of $E(\rho)$, then $s(e)\neq s(f)$.
\end{lemma}

\begin{proof}
  If $s(e)=s(f)$, then any lattice path ordering of $\rho_{\ba e}$ can
  be extended to a lattice path ordering of $\rho$ by adding $e$
  directly following $f$.
\end{proof}

Lemmas \ref{bullseye}, \ref{supportgraph}, 
and \ref{notcontained} treat the case in which,
in a Boolean polymatroid that is an excluded minor, the support of
some element is contained in another.

\begin{lemma}\label{bullseye}
  Let $\rho$ be a Boolean polymatroid that is an excluded minor for
  the class of lattice path polymatroids.  If $e,f,g\in E(\rho)$ and
  $s(e)\subsetneq s(f)\subsetneq s(g)$, then $\rho$ is type $F_1$ or
  $F_3$.
\end{lemma}

\begin{proof}
  No $3$-element excluded minor has such a chain, so $|E(\rho)|\geq 4$
  and $\rho$ has no minors of types $T_1$, $T_2$, $T_3$, or $T_4$.
  Let $C$ be a subset of $E(\rho)$ of maximum size that can be ordered
  so that the support of each is contained in the support of the next.
  Suppose that $|C|\geq 3$.  By Lemma~\ref{distinctsupport}, no two
  supports are equal.  Let $e$ be the element in $C$ with the smallest
  support.  Then $\rho_{\ba e}$ has a lattice path ordering
  $e_1,e_2,\dots ,e_n$ where $s(e_i)=[a_i,b_i]$ for all $i$, and
  $a_1\leq a_2\leq \cdots \leq a_{n}$ and
  $b_1\leq b_2\leq \cdots \leq b_{n}$.  By (S3) in Lemma \ref{lem:S3},
  up to reversing the ordering on $\rho_{\ba e}$, we may assume that
  $s(e)\subsetneq s(e_k) \subsetneq s(e_{k+1})\subsetneq \cdots
  \subsetneq s(e_\ell )$ where $1\leq k < k +|C|-2=\ell \leq n$.  It
  follows that $a_k=a_{k+1}=\cdots =a_\ell$ and
  $b_k<b_{k+1}<\cdots <b_\ell$.  Let $[a,b]$ be the smallest interval
  such that $s(e)\subseteq [a,b]$.  Thus, $a,b\in s(e)$ and
  $a_k\leq a\leq b\leq b_k$.  Also, $s(e)$ might be a proper subset of
  $[a,b]$.

  We first show that $s(e)$ is an interval in some lattice path
  ordering on $\rho_{\del e}$.  To show this, assume that
  $s(e)\ne [a,b]$.  The next two results are the key to being able to
  reorder $[a,b]$ so that $s(e)$ is an interval and the new ordering
  of $E(\rho)$ is a lattice path ordering of $\rho_{\del e}$.

  \begin{sublemma}\label{nost1}
    If $p<k$ and $s(e_p)\cap ([a,b]-s(e))\ne \emptyset$, then
    $[a,b]\subseteq s(e_p)$.
  \end{sublemma}

  Note that $a_p\leq a_\ell\leq a$ and $b<b_\ell$.  If $b_p<b$, then
  $b\notin s(e_p)$, so $\rho |\{e,e_p,e_\ell\}$ is type $T_4$, but
  that is impossible.  So $b\leq b_p$, and the conclusion of
  \ref{nost1} holds.

  \begin{sublemma}\label{nost2}
    If $q>\ell$ and $s(e_q)\cap s(e)\ne\emptyset$, then
    $[a,b]-s(e) \subseteq s(e_q)$.
  \end{sublemma}

  If the conclusion failed, then neither $a$ nor the first element in
  $[a,b]-s(e)$ would be in $s(e_q)$; it would follow that
  $\rho |\{e,e_\ell ,e_q\}$ is type $T_4$, but that is impossible.  So
  \ref{nost2} holds.  

  With \ref{nost1} and \ref{nost2}, we can rearrange $[a,b]$, placing
  $s(e)$ first and $[a,b]-s(e)$ second, to make $s(e)$ into an
  interval without changing the lattice path ordering on
  $\rho_{\ba e}$.  So we may assume that $s(e)$ is the interval
  $[a,b]$.

  If we had $a=a_k$, then $(s(e_{k-1})-s(e))\cap s(e_k)\ne\emptyset$
  (otherwise inserting $e$ between $e_{k-1}$ and $e_k$ would give a
  lattice path ordering for $\rho$, which is impossible), so
  $\rho|\{e_{k-1},e,e_\ell\}$ would have type $T_3$, contrary to
  $\rho$ being an excluded minor.  So $a_k<a$.  Thus, $[a,b]$ cannot
  be moved to the beginning of $s(e_k)$, so either the support of some
  element before $e_k$ blocks it, that is,
  \begin{enumerate}
  \item[(i)] there is a $p<k$ with $[a_k,a)\cap s(e_p)\ne\emptyset$
    and $s(e)\nsubseteq s(e_p)$,
  \end{enumerate}
  or the support of some element after $e_\ell$ blocks it, and so
  \begin{enumerate}
  \item[(ii)] there is a $q>\ell$ with $s(e)\cap s(e_q)\ne\emptyset$ 
    and $a_k\not\in s(e_q)$.
  \end{enumerate}
  Suppose (i) occurs.  If $s(e)\cap s(e_p)\ne\emptyset$, then
  $\rho |\{e,e_p,e_\ell\}$ is type $T_4$, which is impossible.
  Therefore $b_p<a$, and $\rho |\{e,e_p,e_k,e_\ell\}$ is type $F_3$
  (since $s(e)$ and $s(e_p)$ are disjoint, use them as $B$ and $C$).
  Now assume that (ii) occurs.  Since $\rho |\{e,e_\ell ,e_q\}$ is not
  type $T_3$ or $T_4$, we know that $a_q\leq a$ and $b_q=b_\ell$.
  Since $\rho |\{e,e_k,e_q\}$ is not type $T_3$, we know that $a_q=a$
  and $b=b_k$.  Then $\rho |\{e,e_k,e_\ell ,e_q\}$ is type $F_1$,
  where $D=\emptyset$.
\end{proof}

For a Boolean polymatroid $\rho$, we define its \emph{support graph}
$G(\rho)$ to be the graph with vertex set $E(\rho )$ and edge set
$\{ef: e, f\in E(\rho ),\, e\ne f,\, s(e)\cap s(f)\neq \emptyset\}$.
Note that, if $G(\rho)$ is an $n$-cycle for some integer $n\geq 4$,
then $\rho$ is a Boolean $n$-cycle.

\begin{lemma}\label{supportgraph}
  Let $\rho$ be a Boolean polymatroid that is an excluded minor for
  the class of lattice path polymatroids.  Then $G(\rho)$ is
  connected.  Furthermore, if $s(e)\subsetneq s(f)$ for some elements
  $e,f\in E(\rho)$, then $G(\rho_{\ba e})$ is also connected.
\end{lemma}

\begin{proof}
  Suppose $G$ has a component with vertex set $A$ where $E(\rho)-A$ is
  not empty.  Then $\rho |A$ and $\rho |(E(\rho )-A)$ each have
  lattice path orderings, and concatenating these two orderings gives
  a lattice path ordering of $\rho$, which is a contradiction.

  Suppose $s(e)\subsetneq s(f)$, and $G(\rho_{\ba e})$ has distinct
  components $X$ and $Y$.  Then $s(e)$ has a nonempty intersection
  with $s(x)$ and $s(y)$ for some $x\in V(X)$ and $y\in V(Y)$.  Then
  $xfy$ is a path connecting $X$ and $Y$ in $G(\rho_{\ba e})$, which
  is a contradiction.
\end{proof}

We now consider the general case that the support of one element
contains another.

\begin{lemma}\label{notcontained}
  Let the Boolean polymatroid $\rho$ be an excluded minor for the
  class of lattice path polymatroids.  If $s(e)\subsetneq s(f)$ for
  some $e,f\in E(\rho)$, then $\rho$ is type $T_3$, $T_4$, $F_1$,
  $F_2$, or $F_3$.
\end{lemma}

\begin{proof}
  Let $e_1,e_2,\dots ,e_n$ be a lattice path ordering of
  $\rho_{\ba e}$, where $s(e_k)=[a_k,b_k]$ with
  $a_1\leq a_2\leq \cdots \leq a_n$ and
  $b_1\leq b_2\leq \cdots \leq b_n$.  Say $f$ is $e_i$.

  We first show
  \begin{sublemma}\label{nomorev2}
    if $s(e)\subsetneq s(g)$ for some $g\in E(\rho)-\{e,e_i\}$, then
    $\rho$ is type $T_3$, $F_1$, or $F_3$.
  \end{sublemma}
      
  If $\rho$ is not type $F_1$ or $F_3$, then
  Lemma~\ref{bullseye} gives $s(e_i)\nsubseteq s(g)$ and
  $s(g)\nsubseteq s(e_i)$ for all such $g$.  First assume that all
  such $g$ satisfy $s(e)=s(e_i)\cap s(g)$.  By Lemma \ref{lem:S3}, it
  follows that either $e_{i-1}$ or $e_{i+1}$ is such a $g$.  If
  $g=e_{i-1}$, then $s(e)=[a_i,b_{i-1}]$, so
  $e_1,e_2,\dots ,e_{i-1},e,e_i,\dots ,e_{n}$ is a lattice path
  ordering of $\rho$; if $g=e_{i+1}$, then $s(e)=[a_{i+1},b_i]$, so
  $e_1,e_2,\dots ,e_i,e,e_{i+1},\dots ,e_{n}$ is a lattice path
  ordering of $\rho$; both conclusions are impossible, so
  $s(e)\subsetneq s(e_i)\cap s(g)$.  Therefore $\rho |\{e,e_i,g\}$,
  and so $\rho$, is type $T_3$.

  Now assume that if $e_j\in E(\rho)-\{e,e_i\}$, then
  $s(e)\not\subseteq s(e_j)$ .
        
  Next we show that
  \begin{sublemma}
    if $s(e)\cap s(e_j)=\emptyset$ for all $e_j\in E(\rho)-\{e,e_i\}$,
    then $\rho$ is type $F_2$.
  \end{sublemma}
    
  If $i=1$, then $s(e)\subseteq [a_1,a_2)$, so we can reorder this
  interval so that $e,e_1,e_2,\dots , e_{n}$ is a lattice path
  ordering of $\rho$; that is a contradiction, so $i> 1$.  Similarly,
  $i< n$.  By Lemma~\ref{supportgraph}, neither $s(e_{i-1})$ nor
  $s(e_{i+1})$ is disjoint from $s(e_i)$. However, both are disjoint
  from $s(e)$, and so from each other (since $s(e)\subsetneq s(e_i)$),
  so $\rho |\{e,e_{i-1},e_i,e_{i+1}\}$, and so $\rho$, is type $F_2$
  where $s(e_i)$ is $D$.

  Now assume that $s(e)\cap s(e_j)\ne\emptyset$ for some
  $e_j\in E(\rho)-\{e,e_i\}$.
    
  Up to reversing the lattice path ordering of $\rho_{\ba e}$, we may
  assume that $s(e)\cap s(e_{i+1})\ne\emptyset$.  Let the supports of
  $e_i$, $e_{i+1}$, and $e$ be represented by the supports of $e$,
  $f$, and $g$, respectively in Figure~\ref{3venn}.  Then
  $Z\neq\emptyset$ and, since $s(e)$ is properly contained in
  $s(e_i)$, both $V$ and $X$ are empty, and either $T$ or $W$ is
  nonempty.  Now $Y\ne\emptyset$ since $s(e)\nsubseteq s(e_{i+1})$.
  If $T$ and $W$ are both nonempty, then $\rho |\{e,e_i,e_{i+1}\}$,
  and so $\rho$, is type $T_4$.  The remaining cases are where the
  supports of $e_i$, $e_{i+1}$, and $e$ are as shown in the Venn
  diagrams in Figure~\ref{ncproof}, where $U$ may be empty or
  nonempty.

  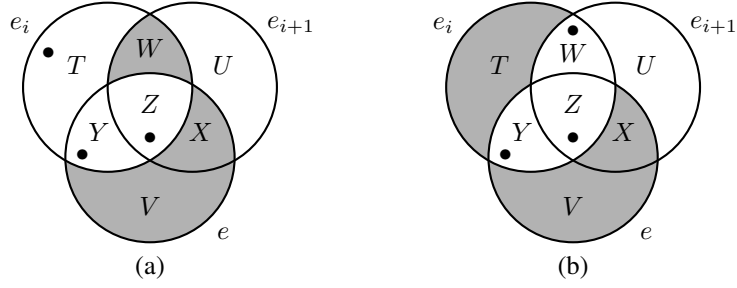
\begin{figure}
\begin{tikzpicture}[scale=0.75]
\begin{scope}
\clip (3.75, 3.5) circle (1.5);
\clip (2.25, 3.5) circle (1.5);
\fill[color= black!30] (0.75,0.75) rectangle (5.25,5); 
\end{scope}
\begin{scope}
\clip (3,2.25) circle (1.5);
\fill[color= black!30] (0.75,0.75) rectangle (5.25,5); 
\end{scope}
 \begin{scope}
 \clip (2.25, 3.5) circle (1.5);
 \clip (3,2.25) circle (1.5);
 \fill[color= white]  (0.75,0.75) rectangle (5.25,5); 
 \end{scope}
  \draw[thick] (2.25,3.5) circle (1.5);
  \draw[thick] (3.75,3.5) circle (1.5);
  \draw[thick] (3,2.25) circle (1.5);
  \node () at (1.7,3.9) {$T$};
  \node () at (1.2,4.1) {$\bullet$};
\node () at (4.3,3.9) {$U$};
\node () at (3,1.4) {$V$};
\node () at (5.5,4.6) {$e_{i+1}$};
\node () at (0.7,4.6) {$e_i$};
\node () at (4.3,0.9) {$e$};
\node () at (3,3.2) {$Z$};
\node () at (3,2.6) {$\bullet$};
\node () at (3,4.2) {$W$};
\node () at (2.1,2.7) {$Y$};
\node () at (1.8,2.3) {$\bullet$};
\node () at (3.9,2.7) {$X$};

\node () at (3,0.3) {(a)};

\begin{scope}
\clip (9.75,3.5) circle (1.5);
\fill[color= black!30] (8.2,0.75) rectangle (12.75,5); 
\end{scope}
\begin{scope}
\clip (10.5,2.25) circle (1.5);
\fill[color= black!30] (8.2,0.75) rectangle (12.75,5); 
\end{scope}
 \begin{scope}
 \clip (9.75,3.5) circle (1.5);
 \clip (10.5,2.25) circle (1.5);
 \fill[color= white] (8.2,0.75) rectangle (12.75,5); 
\end{scope}
 \begin{scope}
 \clip (9.75,3.5) circle (1.5);
 \clip (11.25,3.5) circle (1.5);
 \fill[color= white] (8.2,0.75) rectangle (12.75,5); 
 \end{scope}
  \draw[thick] (9.75,3.5) circle (1.5);
  \draw[thick] (11.25,3.5) circle (1.5);
  \draw[thick] (10.5,2.25) circle (1.5);
\node () at (9.2,3.9) {$T$};
\node () at (11.8,3.9) {$U$};
\node () at (10.5,1.4) {$V$};
\node () at (13,4.6) {$e_{i+1}$};
\node () at (8.2,4.6) {$e_i$};
\node () at (11.8,0.9) {$e$};
\node () at (10.5,3.2) {$Z$};
\node () at (10.5,2.6) {$\bullet$};
\node () at (10.5,4.1) {$W$};
\node () at (10.5,4.5) {$\bullet$};
\node () at (9.6,2.7) {$Y$};
\node () at (9.3,2.3) {$\bullet$};
\node () at (11.4,2.7) {$X$};

\node () at (10.5,0.3) {(b)};

\end{tikzpicture}

    \caption{Venn diagram showing the supports of 
      $e_i,e_{i+1}$, and $e$.}
    \label{ncproof}
  \end{figure}
    
  The case in Figure~\ref{ncproof}(a) has $Z=[a_{i+1},b_i]$.  Now
  $s(e)$ cannot be an interval since otherwise
  $e_1,e_2,\dots ,e_i,e,e_{i+1},\dots ,e_n$ would be a lattice path
  ordering of $\rho$.  Since no reordering of $T\cup Y$ having $T$
  before $Y$ is a lattice path ordering,
  \begin{itemize}
  \item $i>1$, and
  \item there is a $j<i$ with $s(e_j)\cap Y\ne\emptyset$ and
    $T\not\subseteq s(e_j)$.
  \end{itemize}
  Thus, $s(e_j)\cap Z=\emptyset$.  
  If $s(e_j)\cap T\ne\emptyset$, then the proper minor
    $\rho |\{e,e_j,e_i\}$ is type $T_4$, which is impossible.  Thus,
    $s(e_j)\cap T=\emptyset$, and so $\rho |\{e,e_j,e_i,e_{i+1}\}$,
    and hence $\rho$, is type $F_3$.  
    
  The case in Figure~\ref{ncproof}(b) has $Y=[a_i,a_{i+1})$ and
  $W\cup Z=[a_{i+1},b_i]$.  Now $s(e)$ cannot be an interval since
  otherwise $e_1,e_2,\dots ,e_{i-1},e,e_{i},\dots ,e_n$ would be a
  lattice path ordering of $\rho$.  Likewise, we cannot reorder the
  support of $\rho_{\ba e}$ to get a lattice path ordering of the
  support of $\rho$, so either
  \begin{enumerate}
  \item[(i)] there is a $j<i$ with $s(e_j)\cap W\ne \emptyset$ and
    $Z\not\subseteq s(e_j)$, or
  \item[(ii)] there is a $j>i+1$ with $s(e_j)\cap Z\ne \emptyset$ and
    $W\not\subseteq s(e_j)$.
  \end{enumerate}

  Suppose $j<i$ satisfies (i).  Then $Y\subseteq s(e_j)$ and either
  $s(e_j)\cap Z$ is empty and $\rho |\{e,e_j,e_{i+1}\}$ is a Boolean
  $3$-cycle, or $s(e_j)\cap Z$ is not empty and
  $\rho |\{e,e_j,e_{i+1}\}$ is type $T_2$.  This is impossible since
  this is a proper minor of $\rho$.  Thus, (ii) must occur.

  From (ii), we get $s(e_j)\cap Y=\emptyset$.  If
  $s(e_j)\cap W\ne\emptyset$, then the proper minor
  $\rho |\{e,e_i,e_j\}$ is type $T_4$, which is impossible.  Thus,
  $s(e_j)\cap W=\emptyset$.  Then $s(e_j)\cap s(e_i)\subseteq Z$.  The
  proper minor $\rho |\{e,e_{i+1},e_j\}$ is not type $T_1$, so
  $s(e_j)=U\cup Z'$ for some $Z'\subseteq Z$.  Then $U\neq\emptyset$,
  since otherwise $s(e_j)\subsetneq s(e_{i+1})\subsetneq s(e_i)$, and
  $\rho$ is type $F_1$ or $F_3$ by Lemma~\ref{bullseye}, as desired.
  Let $Z''=Z\ba Z'$, which may be empty.  Then
  $s(e)=Y\cup Z'\cup Z''$, and $s(e_i)=Y\cup Z'\cup Z''\cup W$, and
  $s(e_{i+1})=Z'\cup Z''\cup W\cup U$, and $s(e_j)=Z'\cup U$.  The
  only set in $\{U,W,Y,Z',Z''\}$ that may be empty is $Z''$.  If
  $Z''\ne \emptyset$, then the proper minor $\rho |\{e,e_{i+1},e_j\}$
  would be type $T_4$, which is impossible.  Thus, $Z''=\emptyset$,
  and so $\rho |\{e,e_i,e_{i+1},e_j\}$, and hence $\rho$, is type
  $F_1$.
\end{proof}

We prove one last lemma before proving the main result.  

\begin{lemma}\label{onecase}
  Let $\rho$ be a Boolean polymatroid that is an excluded minor for
  the class of lattice path polymatroids.  Let $e_1,e_2,\dots ,e_{n}$
  be a lattice path ordering of $\rho_{\ba e}$ for some
  $e\in E(\rho)$.  If
  $\emptyset\neq s(e_i)\cap s(e_{i+1})\subseteq s(e)\subseteq
  s(e_i)\cup s(e_{i+1})$ for some $i\in\{1,2,\dots ,n\}$, then
  $s(c)\subsetneq s(d)$ for some $c,d\in E(\rho)$.
\end{lemma}

\begin{proof}
  Suppose that $s(c)\not\subseteq s(d)$ for all $c,d\in E(\rho)$.
  Then $a_i<a_{i+1}\leq b_i<b_{i+1}$.  Note that if
  $j\notin\{i,i+1\}$, then
  $[a_{i+1},b_i]\not\subseteq s(e_j)=[a_j,b_j]$, since if $j<i$, then
  $b_j<b_i$, so $b_i\not\in s(e_j)$, while if $j>i+1$, then
  $a_{i+1}<a_j$, so $a_{i+1}\not\in s(e_j)$.

  Let $f$ and $g$ represent $e_i$ and $e_{i+1}$, respectively, in
  Figure~\ref{3venn}.  From the hypothesis, $T\cup X=\emptyset$ and
  $Z\neq\emptyset$.  Sets $W$ and $Y$ are both nonempty since $s(e)$
  is not contained in either $s(e_i)$ or $s(e_{i+1})$.  Sets $U$ and
  $V$ are also nonempty since $s(e)$ contains neither $s(e_i)$ nor
  $s(e_{i+1})$.

  Now $W\cup Y\cup Z$ is not an interval since otherwise
  $e_1,e_2,\dots ,e_i,e,e_{i+1},e_{i+2},\dots ,e_{n}$ would be a
  lattice path ordering on $\rho$, which is false.  Note that
  $[a_i,a_{i+1})=U\cup W$ and $[a_{i+1},b_i]=Z$ and
  $(b_i,b_{i+1}]=V\cup Y$.  Since no reordering of the support of
  $\rho_{\ba e}$ gives a lattice path ordering of $\rho$, and no
  support other than $s(e)$, $s(e_i)$, and $s(e_{i+1})$ contains $Z$,
  there is some element $e_j$ in $E(\rho )-\{e,e_i,e_{i+1}\}$ such
  that either
  \begin{itemize}
  \item[(i)] $j<i$ and $s(e_j) \cap W\ne \emptyset$ and
    $U\not\subseteq s(e_j)$, or
  \item[(ii)] $j>i+1$ and $s(e_j)\cap Y\ne\emptyset$ and
    $V\not\subseteq s(e_j)$.
  \end{itemize}
  Up to reversing the order on $\rho_{\ba e}$, we can assume that (i)
  holds.  Then $s(e_j)\cap (Z\cup Y\cup V)=\emptyset$.  Consider
  $\rho |\{e,e_j,e_i\}$.  Some element in $W$ is in the supports of
  all of these, $a_j<a_i$, and some element in $U$ is in $s(e_i)$ but
  not in $s(e)\cup s(e_j)$ and $Y$ is nonempty and is
    disjoint from $s(e_i)\cup s(e_j)$.  Therefore the proper minor
  $\rho |\{e_j,e_i,e\}$ is type $T_1$, which is a contradiction.
\end{proof}

We now prove the main result.  

\begin{thm}\label{excludedminors}
  A polymatroid is lattice path if and only if it is Boolean and
  contains no minor that is a Boolean $n$-cycle, with $n>2$, or a type
  $T_1$, $T_2$, $T_3$, $T_4$, $F_1$, $F_2$, $F_3$, or $F_4$
  polymatroid.
\end{thm}

\begin{proof} 
  The class of lattice path polymatroids is minor-closed and is
  contained in the class of Boolean polymatroids.  Since none of the
  polymatroids listed above is lattice path, a polymatroid is not
  lattice path if it has any of these as a minor.
  
  If such exists, let $\rho$ be an excluded minor for the class of
  lattice path polymatroids that is a Boolean polymatroid
    but not among those identified above.  Since $\rho$ is Boolean
  and is an excluded minor, by Lemmas \ref{lem:AisUnique}
    and \ref{distinctsupport} each element in $E(\rho)$ can be
  identified by its support.  Since every polymatroid with at most two
  elements is lattice path, $|E(\rho )|\geq 3$.  By
  Lemma~\ref{3element}, $|E(\rho )|\geq 4$.

  Let $G=G(\rho)$.  By Lemma~\ref{supportgraph}, $G$ is connected.
  Let $e$ be a vertex of lowest degree that is not a cut vertex.  Let
  $e_1,e_2,\dots ,e_{n}$ be a lattice path ordering on $\rho_{\ba e}$,
  where $s(e_i)=[a_i,b_i]$ for all $i$ and
  $a_1\leq a_2\leq \cdots \leq a_{n}$ and
  $b_1\leq b_2\leq \cdots \leq b_{n}$.  Since $G\ba e$ is connected
  and Lemma~\ref{notcontained} implies that no support contains
  another, $a_{i-1}<a_i\leq b_{i-1} <b_i$ for each
  $i\in\{2,3,\dots ,n\}$.  Furthermore, $G$ contains $e_1e_2\dots e_n$
  as a path, although this may not be an induced path.  We show that

  \begin{sublemma}\label{deg2}
    the degree $d(e)>1$.
  \end{sublemma}
  
  Suppose $d(e)=1$.  Then $s(e)\cap [a_1,b_n]$ is contained one of the
  following sets:
  \begin{itemize}
  \item[(i)] $[a_1,a_2)$,
  \item[(ii)] $(b_{n-1},b_{n}]$, or
  \item[(iii)] $(b_{i-1},a_{i+1})$ for some $i\in\{2,3,\dots ,n-1\}$.
  \end{itemize}
  These options yield the contradictions (i) $e,e_1,e_2,\dots ,e_n$ is
  a lattice path ordering of $\rho$, (ii) $e_1,e_2,\dots ,e_n,e$ is a
  lattice path ordering of $\rho$, (iii)
  $\rho |\{e,e_{i-1},e_i,e_{i+1}\}$ is type $F_2$.  Thus~\ref{deg2}
  follows.

  Next we show that in the support graph $G$,
  \begin{sublemma}\label{neighbors}
    if $e$ is adjacent to non-adjacent vertices $e_i$ and $e_j$, where
    $i<j$, then $e$ is adjacent to some $e_k$ where $i<k<j$.
  \end{sublemma}
  
  Suppose $e$ is adjacent to $e_i$ and $e_j$ but to none of
  $e_{i+1},e_{i+2},\dots ,e_{j-1}$.  Take a shortest path $P$ between
  $e_i$ and $e_j$ in the graph that $G$ induces on
  $\{e_i,e_{i+1},\dots ,e_j\}$.  Then $P$ together with $e$ forms an
  induced cycle in $G$ with at least four vertices, and
  $\rho |V(P)\cup\{e\}$ is a Boolean $m$-cycle for some integer
  $m\geq 4$.  This contradiction proves~\ref{neighbors}.

  Next we show that
  \begin{sublemma}\label{deg3}
    $d(e)\geq 3$.
  \end{sublemma}
  
  Suppose $d(e)=2$, and let $e_i$ and $e_j$ be adjacent to $e$, where
  $i<j$.  By~\ref{neighbors}, $e_ie_j$ is an edge in $G$.  Let $f$
  represent $e_i$ and $g$ represent $e_j$ in Figure~\ref{3venn}.  If
  $Z=\emptyset$, then $W$, $X$, and $Y$ are all nonempty, so
  $\rho |\{e,e_i,e_j\}$ is a Boolean $3$-cycle, which is a
  contradiction.  Thus, $Z\neq \emptyset$.  Then $Z$ is contained in
  the support of every element in $e_i,e_{i+1},\dots ,e_j$, so
  $j=i+1$.  Since $\rho |\{e,e_i,e_{i+1}\}$ is not type $T_1$ or
  $T_2$, some set in $\{T,U,V\}$ is empty and therefore some set in
  $\{W,X,Y\}$ is empty.  Since no element has its support contained in
  the support of another element, at least one of the following
  statements must hold:
  \begin{enumerate}
  \item[(i)] $T\cup X=\emptyset$ and $U$, $V$, $W$, and $Y$ are all
    nonempty;
  \item[(ii)] $U\cup Y=\emptyset$ and $T$, $V$, $W$, and $X$ are all
    nonempty; or
  \item[(iii)] $V\cup W=\emptyset$ and $T$, $U$, $X$, and $Y$ are all
    nonempty.
  \end{enumerate}
  If (i) occurs, then we get a contradiction by Lemma~\ref{onecase}.
  Up to reversing the order on $\rho_{\ba e}$, we may assume that (ii)
  occurs.  If $i\neq 1$, then $s(e_{i-1})\cap (e_i)\ne\emptyset$ while
  $s(e_{i-1})\cap s(e)=\emptyset$, so $s(e_{i-1})\cap X\ne\emptyset$.
  Then $b_{i-1}\geq a_{i+1}$, and $W\subseteq s(e_{i-1})$, contrary to
  $e$ and $e_{i-1}$ not being adjacent.  So $i=1$ and
  $e,e_1,e_2,\dots ,e_{n-1}$ is a lattice path ordering of $\rho$,
  which is a contradiction.  This completes the proof of~\ref{deg3}.

  Assume that $e$ is adjacent to $e_i$, $e_j$, and $e_k$, with
  $i<j<k$, and to no $e_h$ with $i<h<j$ or $j<h<k$.
  By~\ref{neighbors}, both $e_ie_j$ and $e_je_k$ are edges in $G$.
  Let $f$ represent $e_i$ and $g$ represent $e_j$ in
  Figure~\ref{3venn}.  If $Z=\emptyset$, then $W$, $X$, and $Y$ are
  all nonempty, and $\rho |\{e,e_i,e_j\}$ is a Boolean $3$-cycle,
  which is impossible.  Thus, $Z\ne\emptyset$, i.e.,
  $s(e)\cap s(e_i)\cap s(e_j)\neq\emptyset$.  Thus,
  $s(e)\cap s(e_h)\ne \emptyset$ for all $h$ with $i\leq h\leq j$.
  The same argument shows that
  $s(e)\cap s(e_j)\cap s(e_k)\neq\emptyset$, and so
  $s(e)\cap s(e_h)\ne \emptyset$ for all $h$ with $j\leq h\leq k$.
  Thus, $j=i+1$ and $k=i+2$.

  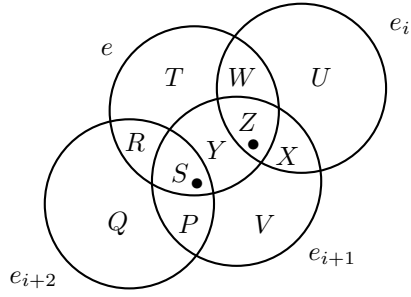
\begin{figure}
\begin{tikzpicture}[scale=0.75]
  \draw[thick] (2.25,3.5) circle (1.5);
  \draw[thick] (4.15,3.9) circle (1.5);
  \draw[thick] (1.1,1.85) circle (1.5);
  \draw[thick] (3,2.25) circle (1.5);
  \node () at (1.9,4.1) {$T$};
  \node () at (4.5,4.1) {$U$};
  \node () at (3.5,1.5) {$V$};
  \node () at (3.2,3.3) {$Z$};
  \node () at (3.1,4.1) {$W$};
  \node () at (2.65,2.8) {$Y$};
  \node () at (2,2.4) {$S$};
  \node () at (0.9,1.5) {$Q$};
  \node () at (2.15,1.5) {$P$};
    \node () at (1.2,2.95) {$R$};
  \node () at (3.9,2.7) {$X$};
  \node () at (5.9,5) {$e_i$};
  \node () at (0.7,4.6) {$e$};
  \node () at (4.7,0.9) {$e_{i+1}$};
  \node () at (-0.6,0.5) {$e_{i+2}$};
  \node () at (3.3,2.9) {$\bullet$};
  \node () at (2.3,2.2) {$\bullet$};
\end{tikzpicture}

\caption{Venn diagram showing the supports of $e$, $e_i$, $e_{i+1}$,
  and $e_{i+2}$.}
    \label{rgv}
  \end{figure}

  We show that
  \begin{sublemma}\label{k4}
    $e_ie_{i+2}$ is an edge in $G$.
  \end{sublemma}

  Suppose not.  Then the supports of $e$, $e_i$, $e_{i+1}$, and
  $e_{i+2}$ are as shown in Figure~\ref{rgv}.  Since $\rho$ has no
  restriction that is type $T_1$, at least one of $Q$, $T\cup W$, and
  $V\cup X$ is empty, as is at least one of $R\cup T$, $U$, and
  $P\cup V$.  Since no restriction of $\rho$ is type $T_2$, either
  $P=\emptyset$ or $R=\emptyset$; also, either $W=\emptyset$ or
  $X=\emptyset$.  No support contains another support, so
  $Q\ne\emptyset$ and $U\ne\emptyset$.  Since
  $s(e)\nsubseteq s(e_{i+1})$, at least one of $T$, $W$, and $R$ is
  nonempty.  Similarly at least one of $P$, $V$, and $X$ is nonempty.
  After possibly reversing the order on $\rho_{\ba e}$, it follows
  that $R$, $T$, $V$, and $X$ are empty, and $P$, $Q$, $S$, $U$, $W$,
  and $Z$ are nonempty.  Now Lemma~\ref{onecase} gives a contradiction
  since
  $\emptyset\neq s(e_i)\cap s(e_{i+1})\subseteq s(e)\subseteq
  s(e_i)\cup s(e_{i+1})$.  This proves~\ref{k4}.

  \begin{figure}
    \begin{tikzpicture}[scale=0.5]
\draw[thick, shift={(-2.5 cm, 0 cm)}, rotate=50] (0,0) ellipse (2.5cm and 4.7cm);
\draw[thick, shift={(2.5 cm, 0 cm)}, rotate=-50] (0,0) ellipse (2.5cm and 4.7cm);
\draw[thick, shift={(0 cm, 1 cm)},  rotate=50] (0,0) ellipse (2.2cm and 4.5cm);
\draw[thick, shift={(0 cm, 1 cm)}, rotate=-50] (0,0) ellipse (2.2cm and 4.5cm);

  \node () at (0,-2.75) {$Q$};
  \node () at (0,-0.75) {$Z$};
  \node () at (0,2.5) {$T$};
  \node () at (-2.25,3.5) {$M$};
  \node () at (2.25,3.5) {$N$};
  \node () at (-3,2.5) {$S$};
  \node () at (3,2.5) {$U$};
  \node () at (-5,1.5) {$L$};
  \node () at (5,1.5) {$O$};
  \node () at (-1.75,0.7) {$X$};
  \node () at (1.75,0.7) {$Y$};
  \node () at (-2.75,-1.2) {$R$};
  \node () at (2.75,-1.2) {$P$};
  \node () at (-0.9,-1.7) {$W$};
  \node () at (0.9,-1.7) {$V$};

  \node () at (-6.6,3.2) {$a$};
  \node () at (-3.25,4.7) {$b$};
  \node () at (3.25,4.7) {$c$};
  \node () at (6.6,3.2) {$d$};
  
 \end{tikzpicture}

  \caption{Venn diagram showing the supports of elements $a$, $b$, $c$,
  and $d$.}\label{4venn}
  \end{figure}
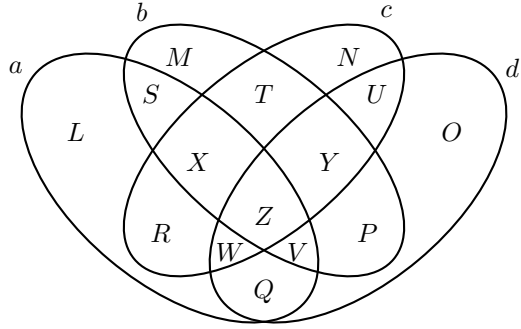

  Let the elements in $\{e,e_i,e_{i+1},e_{i+2}\}$ be represented by
  $\{a,b,c,d\}$ in Figure~\ref{4venn}.  We show that

  \begin{sublemma}\label{sub1}
    $Z=\emptyset$.
  \end{sublemma}

  Suppose $Z\neq\emptyset$.  No restriction of $\rho$ to any set of
  three elements in $\{a,b,c,d\}$ is type $T_1$.  Hence, for each
  triple $\{x,y,z\}\subseteq \{a,b,c,d\}$, at least one set in
  $s(x)-(s(y)\cup s(z))$ and $s(y)-(s(x)\cup s(z))$ and
  $s(z)-(s(x)\cup s(y))$ is empty.  Without loss of generality,
  $s(d)-(s(a)\cup s(b))=O\cup U=\emptyset$, and at least one of
  $L\cup Q$, $M\cup P$, and $N$ is also empty.  Hence, at least one of
  $L\cup Q\cup O$, $M\cup P\cup O$, and $N\cup U\cup O$ is empty.  Up
  to changing the labels of $a$, $b$, $c$, and $d$, we may assume that
  $N\cup O\cup U=\emptyset$.  Since no element has its support
  contained in another by Lemma~\ref{notcontained}, each of the
  following sets is nonempty: $P\cup Y$, $Q\cup W$, $T\cup Y$, and
  $R\cup W$.  Since neither $\rho |\{a,b,c\}$ nor $\rho |\{a,b,d\}$ is
  type $T_2$, both $S\cup V$ and $S\cup X$ are empty.

  Since no support contains another, none of the following sets is
  empty: $L\cup Q$, $M\cup P$, $L\cup R$, $M\cup T$, $R\cup T$, and
  $P\cup Q$.  Since neither $\rho |\{a,c,d\}$ nor $\rho |\{b,c,d\}$ is
  type $T_1$, it follows that at least one of $L,P$, and $T$ is empty
  and at least one of $M,Q$, and $R$ is empty.  Since neither
  $\rho |\{a,c,d\}$ nor $\rho |\{b,c,d\}$ is type $T_2$, it follows
  that at least one of $Q$, $R$, and $Y$ is empty and at least one of
  $P$, $T$, and $W$ is empty.  Thus,
  \begin{enumerate}
  \item[(i)] if $P=\emptyset$, then $R=\emptyset$ and $L$, $M$, $Q$,
    $T$ ,$W$, and $Y$ are all nonempty;
  \item[(ii)] if $T=\emptyset$, then $Q=\emptyset$ and $L$, $M$, $P$,
    $R$, $W$, and $Y$ are all nonempty;
  \item[(iii)] if $L=\emptyset$, then $M\cup W\cup Y=\emptyset$ and
    $P$, $Q$, $R$, and $T$ are all nonempty, and $\rho |\{a,b,c,d\}$
    is type $F_4$, where $A=Q$, $B=R$, $C=Z$, $D=T$, and $E=P$.
  \end{enumerate}

  Since (iii) gives a contradiction and (ii) is obtained from (i) by
  switching the labels on $a$ and $b$, we assume that (i) holds.  Note
  that $b,c,d,a$ is a lattice path ordering of $\rho |\{a,b,c,d\}$ and
  $M,T,Y,Z,W,Q,L$ is a lattice path ordering of the support of
  $\rho |\{a,b,c,d\}$.  Up to reversing the lattice path ordering on
  $\rho_{\ba e}$, where $e_i<e_{i+1}<e_{i+2}$, the fact that $L$, $M$,
  $Q$, $T$, $W$, and $Y$ are all nonempty implies that $(b,c,d,a)$ is
  one of $(e,e_i,e_{i+1},e_{i+2})$, $(e_i,e,e_{i+1},e_{i+2})$,
  $(e_i,e_{i+1},e,e_{i+2})$, and $(e_i,e_{i+1},e_{i+2},e)\}$.

  By Lemma~\ref{onecase}, $e\not\in\{c,d\}$, so $e\in\{a,b\}$.  Then
  $s(e_i)\cap s(e_{i+2})=[a_{i+2},b_i]$ is either $Y\cup Z$ or
  $Z\cup W$, both of which are nonempty.  Then
  $$s(e_i)\cap s(e_{i+2})\subseteq s(e_{i+1})\subseteq s(e_i)\cup
  s(e_{i+2}).$$ We next show that $e_i$ and $e_{i+2}$ are consecutive
  elements in a lattice path ordering of $\rho_{\ba e_{i+1}}$, but
  that is impossible by Lemma~\ref{onecase}, and so that will complete
  the proof of~\ref{sub1}.  The two options for $(b,c,d,a)$ show that
  $s(e_i)\cap s(e_{i+2})\not\subseteq s(e)$.  Now
  $[a_{i+2},b_i]\not\subseteq s(e_{i-1})$ since $b_i>b_{i-1}$.
  Similarly, $[a_{i+2},b_i]\not\subseteq s(e_{i+3})$.  Thus, $e_i$ and
  $e_{i+2}$ are the only elements in $\rho_{\ba e_{i-1}}$ that contain
  $[a_{i+2},b_i]$ in their supports, and so, as claimed, they must be
  consecutive in the lattice path ordering on $\rho_{\ba e_{i+1}}$ by
  Lemma \ref{lem:S3}.  Thus we have shown~\ref{sub1}.

  We show that
  \begin{sublemma}\label{moreempty}
    $s(x)\cap s(y)\cap s(z)=\emptyset$ for every triple
    $\{x,y,z\}\subseteq \{a,b,c,d\}$.
  \end{sublemma}

  Suppose $s(a)\cap s(b)\cap s(c)$ is nonempty.  Then
  $X\neq\emptyset$.  Since $\rho |\{a,b,c\}$ is not type $T_1$, at
  least one of the sets $L\cup Q$, $M\cup P$, and $N\cup U$ is empty.
  Up to relabeling $a$, $b$, and $c$, we may assume that
  $N\cup U=\emptyset$.  Since $\rho |\{a,b,c\}$ is not type $T_2$, at
  least one of the sets $R\cup W$, $S\cup V$, and $T\cup Y$ is empty.
  Two of those
    options yield the contradiction that some support contains
    another, so $S\cup V=\emptyset$; avoiding
    other instances of some support containing another implies that 
  $L\cup Q$, $R\cup W$, $T\cup Y$, and $M\cup P$ are all nonempty.
  Then
  \begin{align*}
    &s(a)=L\cup Q\cup R\cup W\cup X,\\
    &s(c)=R\cup W\cup X\cup T\cup Y,\\
    &s(b)=X\cup T\cup Y\cup M\cup P.
  \end{align*}
  Thus $L\cup Q,R\cup W,X,T\cup Y,M\cup P$ is a lattice path ordering
  of the support of $\rho |\{a,b,c\}$.  Now
  $s(d)= Q\cup W\cup Y\cup P\cup O$.  If $Q\cup W$ and $Y\cup P$ are
  both nonempty, then $\rho |\{a,b,d\}$ is a Boolean 3-cycle, which is
  a contradiction.  Therefore either $Q\cup W$ or $Y\cup P$ is empty.
  Hence either $s(a)\cap s(d)=\emptyset$ or $s(b)\cup s(d)=\emptyset$,
  which is a contradiction.  Then~\ref{moreempty} follows by symmetry.

  Thus, the sets $V$, $W$, $X$, $Y$, and $Z$ in Figure~\ref{4venn} are
  empty.  Since $s(y)\cap s(z)\neq\emptyset$ for all
  $\{y,z\}\subseteq \{a,b,c,d\}$, the following sets are nonempty:
  $P$, $Q$, $R$, $S$, $T$, and $U$.  Therefore $\rho |\{a,b,c\}$ is a
  Boolean $3$-cycle.  This contradiction completes our proof.
\end{proof}

\section*{Acknowledgements}
The AMS Simons Collaboration Grant for Mathematicians (No. 519521)
held by the second author supported her work on this project as well
as the third author's work on this project.

\end{document}